\documentclass[11pt]{article}
\usepackage{cite}
\usepackage{mathrsfs}
\usepackage{amsfonts}
\usepackage{amsmath}
\usepackage{amsfonts,amssymb,color}
\usepackage{dsfont}
\usepackage{curves}
\usepackage{mathrsfs}
\usepackage{pifont}
\usepackage{amssymb}
\usepackage{latexsym,amsmath,amssymb,amsfonts,epsfig,graphicx,cite,psfrag}
\usepackage{eepic,color,colordvi,amscd}
\usepackage{hyperref}

\newtheorem{theorem}{Theorem}[section]

\newtheorem{lemma}{Lemma}[section]
\newtheorem{corollary}{Corollary}[section]

\newtheorem{claim}{Claim}[section]
\newtheorem{conjecture}{Conjecture}[section]

\newcommand{\qed}{\hfill\rule{0.5em}{0.809em}}

\def\emptyset{\mbox{{\rm \O}}}

\newenvironment{proof}{
	\noindent {\bf Proof.}\rm}
{\mbox{}\hfill\rule{0.5em}{0.809em}\par}

\def\qed{\hfill \rule{0.5em}{0.809em}}

\begin{document}
	
	\title{On minimal nonperfectly weight divisible fork-free graphs}
	
	\author{Baogang  Xu\footnote{Email: baogxu@njnu.edu.cn. Supported by 2024YFA1013902},  \; Miaoxia Zhuang\footnote{Corresponding author. Email: 19mxzhuang@alumni.stu.edu.cn }\\\\
		\small Institute of Mathematics, School of Mathematical Sciences\\
		\small Nanjing Normal University, 1 Wenyuan Road,  Nanjing, 210023,  China}
	\date{}

\maketitle

\begin{abstract}
A fork is a graph obtained from $K_{1,3}$ (usually called claw) by subdividing an edge once. A graph $G$ is perfectly weight divisible if for every positive integral weight function on $V(G)$ and each of its induced subgraph $H$, $V(H)$ can be partitioned into $A$ and $B$ such that $H[A]$
is perfect and the maximum weight of
a clique in $H[B]$ is smaller than the maximum weight of a clique in $H$. In this paper, we prove that the perfect weight divisibility of fork-free graphs is equivalent to that of claw-free graphs. We also prove that, for $F\in \{P_7, P_6\cup K_1\}$, each (fork, $F$)-free graph $G$ is perfectly weight divisible and hence $\chi(G)\leq \binom{\omega(G)+1}{2}$.

\begin{flushleft}
\end{flushleft}
	
\end{abstract}

\section{Introduction}
All graphs considered in this paper are finite and simple. Let $G$ be a graph, $v\in V (G)$, and let $X$ and $Y$ be two
subsets of $V (G)$. As usual, we use $N_G(v)$ to denote the set of neighbors of $v$ in $G$, and let $N_G(X)=(\cup_{x\in X}N_G(x))\setminus X$ and $M_G(X)=V(G)\setminus (N_G(X)\cup X)$. If it does not cause any confusion, we usually omit the subscript $G$ and simply write $N(v)$, $M(v)$, $N(X)$, $M(X)$. We say that $v$ is {\em complete} to $X$ if $X\subseteq N(v)$, and say that $v$ is {\em anticomplete} to $X$ if
$X\cap N(v)=\emptyset$. We say that $X$ is complete (resp. anticomplete) to $Y$ if each vertex of X is complete
(resp. anticomplete) to $Y$. If $1 \textless |X| \textless |V (G)|$ and every vertex of $V(G) \setminus X$ is either complete or anticomplete to $X$, then we
call $X$ a {\em homogeneous set}. We use $G[X]$ to denote the subgraph of $G$ induced by $X$, and call $X$ a \textit{clique} if $G[X]$ is a complete
graph. The \textit{clique number} $\omega(G)$ of $G$ is the maximum size taken over all cliques of $G$. For $u,v\in V(G)$, we simply write $u\sim v$ if $uv\in E(G)$ and write $u\not\sim v$ otherwise. For brevity, we let \(N_X(Y)\) (resp. \(M_X(Y)\)) denote \(N_{G[X]}(Y)\) (resp. \(M_{G[X]}(Y)\)).

For an integer $k$, we say that $G$ is $k$-colorable if there is a mapping $\phi$ from $V(G)$ to $\{1,2,...,k\}$ such that $\phi(u)\neq \phi(v)$ whenever $u\sim v$. The chromatic number $\chi(G)$ of $G$ is the smallest integer $k$ such that $G$ is $k$-colorable.

We say that $G$ \textit{induces} $H$ if $G$ has an induced
subgraph isomorphic to $H$, and say that $G$ is \textit{{\em $H$}-free} otherwise. Analogously, for a family ${\cal H}$ of graphs, we say that $G$ is
\textit{${\cal H}$-free} if $G$ induces no member of ${\cal H}$.

A family ${\cal G}$ of graphs is said to be \textit{$\chi$-bounded} if there is a function $f$ such that $\chi(G) \leq f (\omega(G))$ for every $G\in {\cal G}$, and if
such a function does exist for ${\cal G}$, then $f$ is said to be a \textit{binding function} of ${\cal G}$\cite{G1975}. In addition, if $f$ is a
polynomial function then the class ${\cal G}$ is \textit{polynomially $\chi$-bounded}. It has long been known that there are hereditary graph classes that are not $\chi$-bounded (see \cite{CK2024,SR2019,SS2020} for more results and problems on this topic),
 and it is also known that there are hereditary graph classes that are $\chi$-bounded but not
polynomially $\chi$-bounded\cite{BDW2024}.

Let $G$ be a graph, $h$ be a positive integral weight function on $V(G)$ and $X\subseteq V(G)$. We use {\em $\omega_h(G)$} to denote the maximum weight of cliques in $G$ and simply write $\omega_h(G[X])$ as $\omega_h(X)$. Denote $h|_X$ to be a weight function induced on $X$ by $h$. If $\chi(H)=\omega(H)$ for each of its induced subgraph $H$, then $G$ is called a {\em  perfect graph}. An $h$-{\em perfect division} $(A, B)$ of $G$ is a partition of $V(G)$ into $A$ and $B$ such that $G[A]$ is perfect and $\omega_h(B) < \omega_h(G)$ \cite{H2018}. A graph $G$ is $h$-{\em perfectly divisible} if every induced subgraph $H$ of $G$ has an   $h|_{V(H)}$-perfect division. We abbreviate $h$-perfectly divisible as {\em perfectly divisible} if $h$ is the all-ones weight function.  A graph $G$ is {\em perfectly weight divisible} if $G$ is perfectly divisible for every positive integral weight function.



A \textit{hole} of $G$ is an induced cycle of length at least 4, and a $k$-hole is a
hole of length $k$. A $k$-hole is called an \textit{odd hole} if $k$ is odd, and is called an \textit{even hole}
otherwise. An \textit{antihole} is the complement of some hole. An \textit{odd (resp. even) antihole}
is defined analogously. A \textit{fork} is the graph obtained from $K_{1,3}$ (usually called claw) by subdividing an edge once. The class of claw-free graphs is a subclass of fork-free graphs. It is known that there is no linear binding function even for $(3K_1, 2K_2)$-free graphs \cite{BRSV19}, which is a special subclass of claw-free graphs. Kim \cite{K1995} showed that the Ramsey number $R(3,t)$ has order of magnitude $O(\frac{t^2}{logt})$, and thus
for any claw-free graph $G$, $\chi(G) \leq O( \frac{\omega^2(G)}{log\omega(G)})$. Liu {\em et al.} \cite{Yu} proved
that $\chi(G) \leq 7\omega^2(G)$ for every fork-free graph $G$.  Chudnovsky and Seymour proved in \cite{CS2010} that every connected claw-free graph
$G$ with a stable set of size at least three satisfies $\chi(G) \leq 2\omega(G)$.

The authors of \cite{KKS2022} mentioned a conjecture of Sivaraman claiming that fork-free graphs are perfectly divisible.
\begin{conjecture}\label{fork}{\em\cite{KKS2022}}
The class of fork-free graphs is perfectly divisible. 
\end{conjecture}

This conjecture is not even known to be true for claw-free graphs.

Let $G$ and $H$ be two vertex disjoint graphs. We use $G \cup H$ to denote a graph with vertex set $V (G) \cup (H)$ and edge set $E(G) \cup E(H)$, and use $G + H$ to denote a graph with vertex set $V (G) \cup V (H)$ and edge set $E(G) \cup E(H) \cup \{xy | x \in V (G), y \in
V (H)\}$. A \textit{triad} is a stable set with cardinality three. A \textit{dart} is the graph $K_1 + (K_1 \cup  P_3)$, a \textit{banner} is a graph consisting of a $C_4$ with one pendant edge, a \textit{paw} is a graph
obtained from $K_{1,3}$ by adding an edge joining two of its leaves, a \textit{co-dart} is the union of $K_1$ and a paw, a \textit{bull} is a graph
consisting of a triangle with two disjoint pendant edges, a \textit{diamond} is the graph $K_1$ + $P_3$, a \textit{co-cricket} is the union of $K_1$
and a diamond. A {\em balloon} is a graph obtained from a hole $H$ and an additional edge $xy$ such that $x$ has exactly two consecutive neighbors in $H$ and $y$ is anticomplete to $V(H)$. An \textit{$i$-balloon} is a balloon such that its hole has $i$ vertices. An $i$-balloon is
called an \textit{odd balloon} if $i$ is odd. A \textit{parachute} is a graph obtained from a hole $H$ by adding an edge $uv$ such that $u$ is
complete to $V(H)$ and $v$ is anticomplete to $V(H)$. An \textit{$i$-parachute} is a parachute such that its hole has $i$ vertices. An $i$-parachute is called an \textit{odd parachute} if $i$ is odd. See Fig. \ref{fig-1} for these configurations.
\begin{figure}[htbp]
	\begin{center}
		\includegraphics[width=8cm]{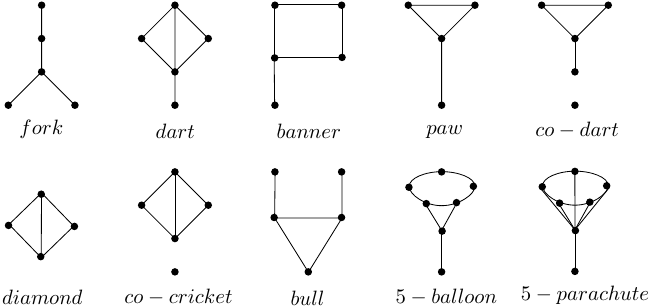}
	\end{center}
	\vskip -25pt
	\caption{Illustration of  fork and some forbidden configurations.}
	\label{fig-1}
\end{figure}

Karthick {\em et al.} \cite{KKS2022} proved that Conjecture \ref{fork} is true on (fork, $F$)-free graphs when $F \in$ \{$P_6$, co-dart,bull\}. They also showed that every $G$ in the classes of (fork, $H$)-free graphs, when $H \in$ \{dart, banner, co-cricket\}, satisfies $\chi(G)\leq \omega^2(G)$ . Wu and Xu\cite{Wu2024} proved that (fork, odd balloon)-free graphs are perfectly divisible (they in fact proved that (fork, odd balloon)-free graphs are perfectly weight divisible). 

Let $G$ be a nonperfectly weight divisible graph. If each of the proper induced subgraphs of $G$ is perfectly weight divisible, then we call $G$ a \textit{minimal nonperfectly weight divisible graph} ({\bf MNPWD} for short). In this paper, we study the structure of minimal nonperfectly weight divisible fork-free graphs, and prove the following theorem.
\begin{theorem}\label{miniclaw}
	Every MNPWD fork-free graph is claw-free.
\end{theorem}

Based on Theorem~\ref {miniclaw}, to confirm Conjecture~\ref {fork}, we may alternatively prove that claw-free graphs are perfectly weight divisible. We also prove that, for $F\in \{P_7, P_{6}\cup K_1\}$, (fork, $F$)-free graphs are perfectly weight divisible, and hence perfectly divisible.

\begin{theorem}\label{P_7}
	Every (fork, $P_7$)-free graph is perfectly weight divisible.
\end{theorem}
\begin{theorem}\label{P6cupK1}
	Every (fork, $P_{6}\cup K_1$)-free graph is perfectly weight divisible.
\end{theorem}

By a simple induction on $\omega(G)$ we have that every perfectly divisible graph $G$ satisfies $\chi(G)\leq \binom{\omega(G)+1}{2}$\cite{H2018}. Combining Theorems \ref{P_7} and \ref{P6cupK1}, we can immediately obtain the following corollary:

\begin{corollary}
	For $F\in \{P_7, P_6\cup K_1\}$, each (fork, $F$)-free graph $G$ satisfies $\chi(G)\leq \binom{\omega(G)+1}{2}$.
\end{corollary}

This paper is organized in the following way. In Section \ref{sec2}, we first present some general properties on minimal nonperfectly weight divisible fork-free graphs. Then, we prove Theorem~\ref{miniclaw} in Section \ref{sec3}, and prove Theorems~\ref{P_7} and \ref{P6cupK1} in Section \ref{sec4}.

\section{Minimal nonperfectly weight divisible fork-free graphs}\label{sec2}

In this section, we study the structure of minimal nonperfectly weight divisible fork-free graphs. The following five lemmas are used often in the sequel.
\begin{lemma}\label{homoset}{\em\cite{CS2019,H2018}}
	No MNPWD graph may have homogeneous sets.
\end{lemma}

\begin{lemma}\label{Mv}{\em\cite{KKS2022}}
Let ${\cal C}$ be a hereditary class of graphs. Suppose that every graph $H \in {\cal C}$ has a vertex $v$ such that $H[M_H(v)]$ is perfect.
Then every graph in ${\cal C}$ is perfectly weight divisible.
\end{lemma}

\begin{lemma}\label{2.3}{\em\cite{Wu2024}}
	Let $G$ be a connected fork-free graph. If $G$ is MNPWD, then for each vertex $v\in V(G)$, $M(v)$ contains no odd antihole except $C_5$.
\end{lemma}

\begin{lemma}\label{2.4}{\em\cite{Wu2024}}
	Let $G$ be a fork-free graph, and $C = v_1v_2 \ldots v_n v_1$ an odd hole of $G$. If there exist two adjacent vertices $u$ and
	$v$ in $V(G)\setminus V(C)$ such that $u$ is not anticomplete to $V(C)$ but $v$ is, then $N(u)\cap V(C) = \{v_j, v_{j+1}\}$, for some $j\in \{1, 2,\ldots,n\}$, or
	$N(u)\cap V(C) = V(C)$.
\end{lemma}

\begin{lemma}\label{2.5}{\em\cite{Wu2024}}
	Let $G$ be an MNPWD fork-free graph, $v\in V(G)$, and $C$ an odd hole contained in $G[M(v)]$.
	Then $N(M(C))$ is not complete to $V(C)$.
\end{lemma}

Let $G$ be an MNPWD fork-free graph, and $v\in V(G)$. By Lemmas~\ref{Mv} and \ref{2.3}, $G[M(v)]$ is not
perfect and must contain an odd hole. Let $C=v_1v_2...v_nv_1$ be an odd hole contained in $G[M(v)]$, where $n\geq 5$ and $n$ is odd.  
In this section, all subscripts are taken modulo $n$. Let   
\begin{eqnarray*}
	U_i(C)&=&\{u\in N(M(C))~|~N_{C}(u)=\{v_i,v_{i+1}\}\}~\mbox{for}~1\leq i\leq n,\\	 
    U'(C)&=&\{u'\in N(M(C))~|~N_{C}(u')=V(C)\},\\
	Z(C)&=&\{z\in N(C)\setminus N(M(C))|z~\mbox{is not complete to}~V(C)\}, and  \\
	Z'(C)&=&\{z\in N(C)\setminus N(M(C))|z~\mbox{is complete to}~V(C)\}
    \end{eqnarray*}
Let $U(C)=\bigcup_{i=1}^n U_i(C)$, $Y(C)=N_{M(C)}(U(C))$, and let $Y'(C)$ be the union of the vertex sets of all components of $G[M(C)]$ that contain some vertices of $Y(C)$. By Lemma \ref{2.4}, we can partition $N(M(C))$ into $U(C)$ and $U'(C)$, and $U(C)\neq \emptyset$ by Lemma \ref{2.5}. Note that $N(C)=Z'(C)\cup Z(C)\cup U(C)\cup U'(C)$.

If $U_i$ has two nonadjacent vertices $u$ and $u'$, then $\{u, u', v_i, v_{i-1}, v_{i-2}\}$ induces a fork. Therefore,
\begin{equation}\label{eqa-Ui-clique}
\mbox{$U_i(C)$ is a clique, and $U_i(C)\cap U_j(C)= \emptyset$ for any $i,j\in \{1,2,\ldots,n\}$ and $i\neq j$.}
\end{equation}

\begin{figure}[htbp]
	\begin{center}
		\includegraphics[width=5cm]{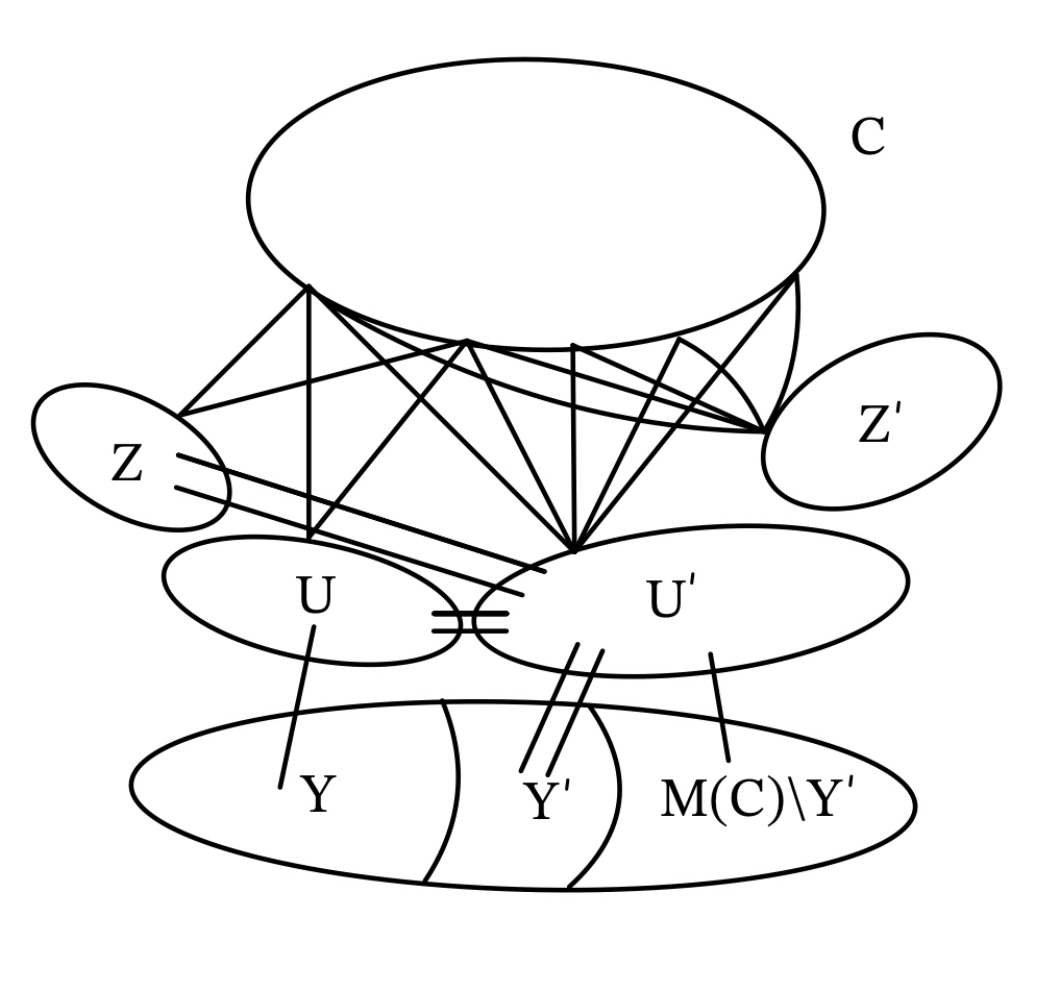}
	\end{center}
	\vskip -25pt
	\caption{A simple illustration of the sets $V(C), U, U', Z, Z'$ and $M(C)$ etc.}
	\label{fig-2}
\end{figure}

Throughout the remainder of this section, suppose that $G$ is a minimal nonperfectly weight divisible fork-free graph, $v\in V(G)$, and $C=v_1v_2\ldots v_nv_1$ is an odd hole contained in $G[M(v)]$. When there is no ambiguity, we abbreviate $U(C),U'(C),Y(C),Y'(C),Z(C),Z'(C),U_i(C)$ as  $U,U',Y,Y',Z,Z',U_i$ respectively, where $i\in\{1,2,\ldots,n\}$.  See Figure~\ref{fig-2}.

The following lemmas describe the sets $V(C),U,U',Z,Z'$, and $M(C)$ and the relations among them.

\begin{lemma}\label{lem-NMC(u)-clique}
	 $N_{M(C)}(u)$ is a clique for any $u\in U$.
\end{lemma}
\begin{proof}
	Suppose to its contrary, let $u\in U$ such that $N_{M(C)}(u)$ has two nonadjacent vertices $a_1$ and $a_2$. Without loss of generality, suppose that $N_C(u)=\{v_1,v_2\}$. Then \{$v_n,v_1,u,a_1,a_2$\} induces a fork.
\end{proof}
\begin{lemma}\label{lem-U'-MC}
For every $u'\in U'$ and every component $B$ of $G[M(C)]$, $u'$ is either complete or anticomplete to $V(B)$.
\end{lemma}
\begin{proof}
	 Let $u'\in U'$ and $a'\in N_{M(C)}(u')$. For any $a''\in N_{M(C)}(a')$, we have that $a''\sim u'$ to forbid \{$v_{n-1},v_1,u',a',a''$\} to be a fork.
\end{proof}
\begin{lemma}\label{lem-MC-Y'}
	$M(C)\setminus Y'$ is a stable set, and $N(x)\subseteq U'$ for any $x\in M(C)\setminus Y'$.
\end{lemma}
\begin{proof} Let $B$ be a component of $G[M(C)\setminus Y']$. By the definition of $Y'$, $V(B)$ is anticomplete to $V(G)\setminus U'$. By Lemma \ref{lem-U'-MC}, each vertex of $U'$ is either complete or anticomplete to $V(B)$. If $|V(B)|\ge 2$, then $V(B)$ is a homogeneous set of $G$, contradicting Lemma \ref{homoset}. Therefore, $|V(B)|=1$, which means that $M(C)\setminus Y'$ is a stable set, and thus $N(x)\subseteq U'$ for any $x\in M(C)\setminus Y'$. This proves Lemma~\ref{lem-MC-Y'}. 
\end{proof}

\begin{lemma}\label{lem-U'-com}
	$U'$ is complete to $U\cup V(C)\cup Y'\cup Z$.
\end{lemma}
\begin{proof}
It follows from its definition that $U'$ is complete to $V(C)$.

Suppose that $U'$ is not complete to $Z$. Let $z\in Z$ and $u'\in U'$ be nonadjacent. Choose $a\in N_{M(C)}(u')$, and choose $v_i\in N_C(z)$ and $v_j\in V(C)\setminus N_C(z)$ that are nonadjacent on $C$. Then $\{a,u',v_i,v_j,z\}$ induces a fork. Hence $U'$ is complete to $Z$.

By Lemma~\ref{lem-U'-MC}, to prove that $U'$ is complete to $Y'$, it is enough to prove that it is complete to $Y$. Choose $y\in Y$ and $u'\in U'$ are nonadjacent. Choose $u\in U_i$ with $y\in N_{M(C)}(u)$. Let $a'\in N_{M(C)}(u')$. If $a'\sim u$, then Lemma~\ref{lem-NMC(u)-clique} gives $a'\sim y$, and Lemma~\ref{lem-U'-MC} then gives $u'\sim y$, a contradiction. Thus $a'\not\sim u$. If $u\not\sim u'$, then $\{u,v_{i+1},u',v_{i-1},a'\}$ induces a fork; if $u\sim u'$, then $\{y,u,u',a',v_{i-1}\}$ induces a fork. Therefore $U'$ is complete to $Y$, and hence to $Y'$.

Finally, suppose that $u\in U_i$ and $u'\in U'$ are nonadjacent. Since $u'$ is complete to $Y'$, it is complete to $N_{M(C)}(u)$. Choose $y\in N_{M(C)}(u)$. Then $\{u,y,u',v_{i-1},v_{i+2}\}$ induces a fork. Thus $U'$ is complete to $U$. This completes the proof of Lemma \ref{lem-U'-com}.
\end{proof}

\begin{lemma}\label{lem-Z'-clique-anticomtoU}
$Z'$ is a clique, and is anticomplete to $U$.
\end{lemma}
\begin{proof}
Suppose that $z'\in Z'$ is adjacent to some $u_i\in U_i$, and choose $y_i\in N_{M(C)}(u_i)$. Then $\{y_i,u_i,z',v_{i+2},v_{i-1}\}$ induces a fork. Hence $Z'$ is anticomplete to $U$.

Now suppose that $z_1',z_2'\in Z'$ are nonadjacent. Since $U\neq\emptyset$, choose $u_j\in U_j$ and $y_j\in N_{M(C)}(u_j)$. Then $\{y_j,u_j,v_j,z_1',z_2'\}$ induces a fork. Therefore $Z'$ is a clique. 
\end{proof}

\begin{lemma}\label{lem-Z-ViVi+1}
For every $z\in Z$, there is an $i\in\{1,2,\ldots,n\}$ such that $\{v_i,v_{i+1}\}\subseteq N_C(z)$.
\end{lemma}
\begin{proof}
Let $z\in Z$. If $z$ has a unique neighbor $v_i$ on $C$, then $\{z,v_i,v_{i-1},v_{i+1},v_{i+2}\}$ induces a fork. Hence $|N_C(z)|\geq2$.
	
Suppose that $\{v_i, v_{i+1}\}\not\subseteq N_C(z)$ for every $i$. Without loss of generality, let $v_1\sim z$. 
Then $z\not\sim v_2,v_n$. If $z\not\sim v_3$, then \{$v_3, v_2, v_1, v_n, z$\} induces a fork. Thus $z\sim v_3$. Since $z\not\sim v_4$, we must have $z\sim v_5$; otherwise, $\{v_5,v_4,v_3,v_2,z\}$ induces a fork. If $n=5$, this contradicts $z\not\sim v_n$. We may therefore assume that $n\geq7$. Continuing in this way, $z\sim v_j$ for every $j\in\{3,5,7,\ldots,n-2\}$ and $z\not\sim v_k$ for every $k\in\{2,4,6,\ldots,n-1\}$. In particular, $z\not\sim v_{n-3},v_{n-1}$, and then $\{v_{n-3},v_{n-2},v_{n-1},v_n,z\}$ induces a fork, a contradiction. This proves Lemma~\ref{lem-Z-ViVi+1}.
\end{proof}

The following lemmas show more relations between $U$ and $Z$.

\begin{lemma}\label{lem-ZU-comUiViVi+1}
Let $z\in Z\cap N(U)$. If $N_{U_i}(z)\neq\emptyset$, then $z$ is complete to $U_i\cup\{v_i,v_{i+1}\}$.
\end{lemma}
\begin{proof}
Let $u\in N_{U_i}(z)$. First, we prove that $z$ is complete to $\{v_i,v_{i+1}\}$. Without loss of generality, suppose to its contrary that $z\not\sim v_i$. 
To forbid a fork on $\{v_i,v_{i-1},u,y,z\}$, we have $z\sim v_{i-1}$. 
Then, to forbid a fork on $\{v_{i+1},v_{i-1},u,y,z\}$, we have $z\sim v_{i+1}$. If $z\sim v_j$ where $j\in \{i+2,i+3,...,i-3\}$,
then $\{v_{j},v_{i-1},u,z,y\}$ induces a fork, a contradiction.
By Lemma \ref{lem-Z-ViVi+1}, $N_C(z)=\{v_{i+1},v_{i-2},v_{i-1}\}$, which forces a fork on $\{v_{i-2},z,u,y,v_{i}\}$, a contradiction. Consequently, $z$ is complete to $\{v_i,v_{i+1}\}$.

Now, we prove that $z$ is complete to $U_i$. Let $u_1, u_2\in U_i$ such that $u_1\sim z$ and $u_2\not\sim z$. Let $y_1\in N_{M(C)}(u_1)$, $y_2\in N_{M(C)}(u_2)$. If $z\sim v_3$, then \{$y_1,u_1,z,v_3,v_n$\} induces a fork whenever $z\sim v_n$, and \{$v_n,v_1,u_1,z,v_3$\} induces a fork whenever $z\not\sim v_n$. If $z\not\sim v_3$, then \{$y_2,u_2,v_2,v_3,z$\} induces a fork. All are contradictions. This completes the proof of Lemma \ref{lem-ZU-comUiViVi+1}.
\end{proof}

\begin{lemma}\label{lem-ZU-NCz}
For any $i$ and $z\in Z\cap N(U_i)$, $N_C(z)=N_C(u_i)$, or $N_C(z)=N_C(u_i)\cup \{v_j, v_{j+1}\}$ for some $j\ne i$.
\end{lemma}
\begin{proof}
Let $z\in Z\cap N(U_i)$, $u_i\in U_i$, and $a\in N_{M(C)}(u_i)$. Then, $z$ is complete to $\{v_i, v_{i+1}\}\cup U_i$ by Lemma~\ref{lem-ZU-comUiViVi+1}. If $|N_C(z)|\geq 5$, let $v_{j_1}$, $v_{j_2}\in N_C(z)\setminus N_C(u_i)$ such that $|j_1-j_2|\geq 2$, then \{$a,u_i,z,v_{j_1},v_{j_2}$\} induces a fork, a contradiction. Therefore, $|N_C(z)|\le 4$. If $N_C(z)\subseteq \{v_k,v_{k+1},v_{k+2}\}$ for some $k$, then conclusion follows trivially. If $N_C(z)=\{v_i,v_{i+1}, v_j\}$ where $j\not\in\{i-1,i+2\}$, then \{$u_i,z,v_j,v_{j-1},v_{j+1}$\} induces a fork, a contradiction. Suppose that $N_C(z)=\{v_i,v_{i+1},v_j,v_k\}$. To forbid a fork on $\{a,u_i,z,v_j,v_k\}$, we have $|j-k|= 1$.
This completes the proof of Lemma~\ref{lem-ZU-NCz}.
\end{proof}

As a consequence of Lemmas~\ref{lem-ZU-comUiViVi+1} and \ref{lem-ZU-NCz}, we have the following corollary.
\begin{corollary}\label{co-ZU-NUz}
	Let $z\in Z\cap N(U)$. Then, $N_U(z)\subseteq U_i\cup U_j$ for some $i,j\in\{1,2,..,n\}$.
\end{corollary}
\begin{proof} By Lemmas~\ref{lem-ZU-comUiViVi+1} and \ref{lem-ZU-NCz}, the statement is certainly true if $|N_C(z)|\le 3$. Suppose so that $|N_C(z)|=4$, and let $N_C(z)=N_C(u_i)\cup \{v_j, v_{j+1}\}$ for some $j\not\in \{i-1, i, i+1\}$. If $j\not\in \{i-2, i+2\}$, then $N_U(z)\subseteq U_i\cup U_j$ by Lemma~\ref{lem-ZU-comUiViVi+1}. Otherwise,
we suppose by symmetry that $j=i+2$. If $N_U(z)\setminus(U_i\cup U_{i+2})\neq\emptyset$, then $z$ has a neighbor, say $u'$, in $U_{i+1}$, and so $\{y, u', z, v_i, v_{i+3}\}$ induces a fork for any $y\in N_{M(C)}(u')$. This proves Corollary~\ref{co-ZU-NUz}.
\end{proof}

\begin{lemma}\label{lem-U-NMCuiuj}
Let $u_i\in U_i$ and $u_j\in U_j$ for $1\le i<j\le n$. If $u_i\sim u_j$, then $N_{M(C)}(u_i)=N_{M(C)}(u_j)$.
\end{lemma}
\begin{proof}
If $u_i\sim u_j$ and $y\in N_{M(C)}(u_i)\setminus N_{M(C)}(u_j)$, then either \{$y, u_i, v_i, u_j, v_j$\} (when $j\ne i+1$) or \{$y, u_i, v_i, u_j, v_{j+1}$\} (when $j=i+1$) induces a fork, a contradiction. Similar contradiction happens if $u_i\sim u_j$ and $N_{M(C)}(u_j)\setminus N_{M(C)}(u_i)\ne\emptyset$.
\end{proof}

\begin{lemma}\label{lem-W}
Suppose $U_i\neq \emptyset$ for some $i\in\{1,2,...,n\}$, and let $W=N(U_i\cup M(C))$. Then $\omega_h(W)<\omega_h(G)$ for any positive integral weight function $h$.
\end{lemma}
\begin{proof} In fact, we prove that every maximal clique of $G[W]$ is contained in a larger maximal clique of $G$. Suppose to its contrary that $X\subseteq W$ is a maximal clique of $G$. 
Since $Z'$ is anticomplete to $U$ by Lemma~\ref{lem-Z'-clique-anticomtoU}, we have that $W= \{v_i,v_{i+1}\}\cup (U\setminus U_i) \cup (Z\cap N(U_i))\cup U'$. Note that $U'$ is complete to $V(C)\cup U\cup Z\cup Y'$ by Lemma~\ref{lem-U'-com}.
 
First, we prove that 
\begin{equation}\label{eqa-W1}
    X\cap (Z\cap N(U_i))\ne \emptyset.
\end{equation}
	
Suppose that (\ref{eqa-W1}) does not hold. Then,  $X\subseteq (U\setminus U_i)\cup \{v_i,v_{i+1}\}\cup U'$. If $X\subseteq U_j\cup U'$ for some $j\ne i$, then $X$ is complete to \{$v_j, v_{j+1}$\} where $\{v_j, v_{j+1}\}\not\subseteq W$, a contradiction.
If $X\not\subseteq U_l\cup U'$ for any $l\ne i$, then $X$ is complete to $N_{M(C)}(u)$ where $u\in X$, since for any two distinct integers $j$ and $k$ and any $u_j\in U_j\cap X$ and $u_k\in U_k\cap X$, $\{u_j, u_k\}$ is complete to $N_{M(C)}(u_k)$ by Lemma~\ref{lem-U-NMCuiuj}. Therefore, $X$ must be complete to $N_{M(C)}(u_k)$, a contradiction. 

Consequently, $X\not\subseteq (U\setminus U_i)\cup U'$. It is obvious that $X\not\subseteq \{v_i,v_{i+1}\}\cup U'$ since $\{v_i,v_{i+1}\}\cup U'$ is complete to $U_i$ and $U_i\ne \emptyset$. Suppose that $X\cap(U\setminus U_i)\cup U'\ne \emptyset$ and $X\cap \{v_i,v_{i+1}\}\ne \emptyset$. If $\{v_i,v_{i+1}\}\subseteq X$, then $X\cap U\setminus U_i=\emptyset$ since $u$ is not complete to $\{v_i,v_{i+1}\}$ for any $u\in U\setminus U_i$, and so $X\subseteq \{v_i,v_{i+1}\}\cup U'$, contradicting $X\not\subseteq \{v_i,v_{i+1}\}\cup U'$. So we can suppose by symmetry that $v_i\in X$ and $v_{i+1}\not\in X$. Then $X\subseteq U_{i-1}\cup \{v_i\}\cup U'$ which is complete to $\{v_{i-1}\}$, a contradiction.
This proves (\ref{eqa-W1}).

Next, we prove that 
\begin{equation}\label{eqa-W2}
    X\cap\{v_i,v_{i+1}\}\ne \emptyset.    
\end{equation}
Suppose to its contrary that $X\subseteq (U\setminus U_i)\cup (Z\cap N(U_i))\cup U'$. By Corollary~\ref{co-ZU-NUz}, $X\subseteq U_j\cup Z\cap N(U_i))\cup U'$ for some $j\ne i$. Therefore, $X$ is complete to $\{v_j,v_{j+1}\}\not\subseteq W$ by Lemma~\ref{lem-ZU-comUiViVi+1}, a contradiction. This proves (\ref{eqa-W2}).

Note that $\{v_i,v_{i+1}\}\cup U'\cup (Z\cap N(U_i))$ is complete to $U_i$ by Lemmas~\ref{lem-U'-com} and \ref{lem-ZU-comUiViVi+1}, we have that $X\not\subseteq \{v_i,v_{i+1}\}\cup U'\cup (Z\cap N(U_i))$. Thus, $X\cap (U\setminus U_i)\ne\emptyset$ and so $|X\cap\{v_i,v_{i+1}\}|=1$ by (\ref{eqa-W2}). Suppose by symmetry that $v_i\in X$, $v_{i+1}\not\in X$ and $z\in Z\cap N(U_i)$ by (\ref{eqa-W1}). 
 We have $X\subseteq N_W(z)\subseteq U_{i-1} \cup (Z\cap N(U_i))\cup \{v_{i}\}\cup U'$, and so is complete to $\{v_{i-1}\}$ by Lemma~\ref{lem-ZU-comUiViVi+1}, a contradiction. This completes the proof of Lemma \ref{lem-W}.
\end{proof}

The following conclusion is very useful, particularly, to the proofs of Theorems~\ref{P_7} and \ref{P6cupK1} in Section \ref{sec4}.

\begin{lemma}\label{lem-u-mix}
There exist a vertex $v\in V(G)$, an odd hole $C$ contained in $G[M(v)]$, a component $B$ of $G[M(C)]$, and $u\in U$ such that $u$ is mixed on $V(B)$.
\end{lemma}
\begin{proof}
Let $C$ an odd hole of $G$ with $M(C)\ne\emptyset$. Suppose that for any $v\in M(C)$, any component $B$ of $G[M(C)]$, and any $u\in U$, $u$ is not mixed on $V(B)$. Without loss of generality, we can suppose that $U_1\neq \emptyset$ by Lemma~\ref{2.5}.

By Lemma~\ref{lem-U'-MC}, each vertex of $U'$ is either complete or anticomplete to $V(B)$. If $|V(B)|\ge 2$, then $V(B)$ is a homogeneous set of $G$ as $V(B)$ is anticomplete to $V(G)\setminus(U\cup U')$, a contradiction to Lemma~\ref{homoset}. Therefore, $|V(B)|=1$, and hence
\begin{equation}\label{eqa-stableset-C}
\mbox{ $M(C)$ is a stable set.}
\end{equation}

Since $U_1$ is a clique by (\ref{eqa-Ui-clique}), we have that $G[U_1\cup M(C)]$ is perfect.
If $G[M(U_1\cup M(C))]$ is perfect, let $W=U_1\cup M(C)\cup M(U_1\cup M(C))$, then by Lemma~\ref{lem-W}, $(W, V(G)\setminus W)$ is a perfect weighted division of $G$, contradicting the minimal nonperfect weighted divisibility of $G$. Therefore, $G[M(U_1\cup M(C))]$ is not perfect. By Lemma~\ref{2.3}, we can choose an odd hole $C'$ in $G[M(U_1\cup M(C))]$. But now, $M(C')$ is not a stable set as $U_1\cup M(C)\subseteq M(C')$ and $\omega(U_1\cup M(C))\ge 2$. By  (\ref{eqa-stableset-C}), there must exist a $v'\in M(C')$, a component $B'$ of $G[M(C')]$, and $u'\in U(C')$ such that $u'$ is mixed on $V(B')$. This completes the proof of Lemma \ref{lem-u-mix}.
\end{proof}


\section{Proof of Theorem~ \ref{miniclaw}}\label{sec3}

In this section, we prove Theorem~\ref{miniclaw}. Firstly, we need to prove that MNPWD fork-free graphs must be odd parachute-free. Recall that an odd parachute is obtained from an odd hole $H$ and an edge $uv$ by joining $u$ to all vertices of $H$. 

All the notations $U, U', Y, Y', Z$ and $Z'$ are defined the same as in Section \ref{sec2}. If $G$ is an MNPWD fork-free graph, then by Lemma~\ref{2.5}, for any odd hole $C$ of $G$, we have  that $U\neq \emptyset$. Without loss of generality, we suppose that $U_1\neq \emptyset$.

\begin{lemma}\label{lem-U'-MCstable}
Let $G$ be an MNPWD fork-free graph. Suppose that $G$ contains an odd parachute $F$, and let $C$ be the odd hole of $F$. Then, $M(C)$ is a stable set.
\end{lemma}
\begin{proof} Let $C=v_1v_2...v_nv_1$, and let $v$ be the vertex of degree 1 of $F$. First we have that $M(C)\neq \emptyset$ and $U'\neq \emptyset$.

By Lemma~\ref{lem-MC-Y'}, we have that $M(C)\setminus Y'$ is a stable set, and $N(x)\subseteq U'$ for any $x\in M(C)\setminus Y'$. By Lemma~\ref{lem-U'-com}, $U'$ is complete to $U\cup V(C)\cup Y'\cup Z$. If $Z'=\emptyset$, or if $Z'$ is nonempty but complete to $U'$,  then $V(G)\setminus (U'\cup M(C)\setminus Y')$ is a homogeneous set, contradicting Lemma~\ref{homoset}. Therefore,
\begin{equation}\label{eqa-U'-Z'}
\mbox{$Z'\neq \emptyset$, and $Z'$ is not complete to $U'$.}
\end{equation}

Next, we show that $Y'=Y$ and $|Y|=1$. Then, it follows that $M(C)$ is a stable set since $M(C)\setminus Y'$ is stable by Lemma~\ref{lem-MC-Y'}.

Suppose that $Y'\ne Y$, and let $y'\in Y'\setminus Y$. Let $u_i\in U_i$ for some $i$. It is certain that $u_i\not\sim y'$. By (\ref{eqa-U'-Z'}), there exist $u'\in U'$ and $z'\in Z'$ such that $u'\not\sim z'$. Since $U'$ is complete to $U\cup Y'$  by Lemma~\ref{lem-U'-com}, we have that $u'\sim u_i$ and $u'\sim y'$. Then, \{$z', v_{i+2}, u', u_i, y'$\} induces a fork. Therefore, $Y'= Y$.

If $U$ is not complete to $Y$, let $u\in U_j$ for some $j$ and $y''\in Y$ such that $u\not\sim y''$, then \{$z', v_{j+2}, u', u, y''$\} induces a fork. Therefore, $U$ is  complete to $Y$. Now, $Y$ is complete to $U\cup U'$, and anticomplete to $V(G)\setminus (U\cup U'\cup Y)$. Since $G$ has no homogeneous sets by Lemma~\ref{homoset}, we have that $|Y|=1$. This completes the proof of Lemma~\ref{lem-U'-MCstable}.
\end{proof}

\begin{lemma}\label{lem-MNPD-parachute}
Every MNPWD fork-free graph is odd parachute-free.
\end{lemma}
\begin{proof}
Suppose to its contrary, let $G$ be an MNPWD fork-free graph that contains an odd parachute $F$. Let $C=v_1v_2...v_nv_1$ be the odd hole contained of $F$ such that $n\geq 5$ and $n$ is odd, and $v$ be the vertex of degree 1 of $F$.
\begin{equation}\label{eqa-para-1}
	G[M(U_1\cup M(C))] \mbox{ is perfect}.
\end{equation}
Notice that $M(U_1\cup M(C))\subseteq V(C)\cup Z\cup Z'$. Suppose that $G[M(U_1\cup M(C))]$ is not perfect. By Lemma~\ref{2.3}, we may choose $H$ to be an odd hole contained in $G[M(U_1\cup M(C))]$. Since $Z'$ is a clique by Lemma~\ref{lem-Z'-clique-anticomtoU}, we have that $|V(H)\cap Z'|\leq 2$. Let $u'\in U'$. Since $U'$ is complete to $V(C)\cup U\cup Z\cup Y'$ by  Lemma~\ref{lem-U'-com}, and since
$M(U_1\cup M(C))\subseteq V(C)\cup Z\cup Z'$, we have that $u'$ has at least 3 neighbors in $M(H)$. It follows from Lemma~\ref{2.4} that $u'\in U'(H)$, which $G$ has an odd parachute containing the odd hole $H$. However, $U_1\cup M(C)\subseteq M(H)$, which implies that $M(H)$ cannot be a stable set, contradicting Lemma~\ref{lem-U'-MCstable}. This proves (\ref{eqa-para-1}).

Since $U_1$ is a clique by (\ref{eqa-Ui-clique}), and since $M(C)$ is a stable set by Lemma~\ref{lem-U'-MCstable}, we have that $G[U_1\cup M(C)]$ is perfect. Combining this with (\ref{eqa-para-1}), we have that $G[U_1\cup M(C)\cup M(U_1\cup M(C))]$ is perfect. Let $A=U_1\cup M(C)\cup M(U_1\cup M(C))$ and $B=N(U_1\cup M(C))$. By Lemma~\ref{lem-W}, $A$ and $B$ form a perfect weighted division of $G$ for any positive integer function, a contradiction. This completes the proof of Lemma~\ref{lem-MNPD-parachute}.
\end{proof}

As a consequence of Lemmas~\ref{2.4} and \ref{lem-MNPD-parachute}, we have the following corollary.
\begin{corollary}\label{new2.5}
Let $G$ be an MNPWD fork-free graph, and $C=v_1v_2...v_nv_1$ an odd hole of $G$. Then for any $u\in N(M(C))$, $N_C(u) = \{v_j, v_{j+1}\}$ for some $j\in \{1, 2,\ldots,n\}$.
\end{corollary}

Recall that for each odd hole $C$ of $G$, we have associated sets $U, U', Y, Y', Z$ and $Z'$ as defined in Section 2.  We call a vertex a {\em claw center} if it is the center of an induced claw.

\begin{lemma}\label{lem-claw}
	Let $G$ be an MNPWD fork-free graph. Suppose that $G$ contains a claw with center $v$. Then, $G[M(v)]$ is perfect.
\end{lemma}
\begin{proof}
	Suppose to its contrary that $G$ contains a claw with center $v$ and $G[M(v)]$ is not perfect. By Lemma~\ref{2.3}, we may choose $C=v_1v_2...v_nv_1$ to be an odd hole of $G[M(v)]$.
By Corollary~\ref{new2.5}, $U'=\emptyset$, and thus $N(M(C))=U$.

\begin{claim}\label{claw1}
For any $y\in Y$, $y$ is not a claw center.
\end{claim}
\begin{proof}
	Suppose to its contrary. Let $\{a_1,a_2,a_3,y\}$ be an induced claw of $G$ with center $y\in Y$. Then, $\{a_1,a_2,a_3\}\subseteq M(C)\cup U$.
	
	First, suppose that \{$a_1,a_2,a_3\}\subseteq M(C)$. Let $u\in U_i$ be a neighbor of $y$ for some $i\in \{1,2,...,n\}$. Since $N_{M(C)}(u)$ is a clique by Lemma~\ref{lem-NMC(u)-clique}, we have that $u$ has at most one neighbor in \{$a_1,a_2,a_3$\}. Without loss of generality, suppose that $u\not\sim a_1$ and $u\not\sim a_2$. Then \{$v_i,u,y,a_1,a_2$\} induces a fork, a contradiction.
	
	Next, suppose that \{$a_1,a_2$\}$\subseteq M(C)$ and $a_3\in U$. Let $v_i\in  N_C(a_3)$. Then \{$v_i,a_3,y,a_1,a_2$\} induces a fork, a contradiction.
	
	Thirdly, suppose that $a_1\in M(C)$ and $\{a_2,a_3\}\subseteq U$. Without loss of generality, let $a_2\in U_1$. Then $a_3$ is complete to \{$v_1,v_2$\} to forbid a fork on \{$v_1,a_2,y,a_1,a_3$\} or \{$v_2,a_2,y,a_1,a_3$\}, which forces an induced fork on \{$v_4,v_3,v_2,a_2,a_3$\}, a contradiction.
	
	Fourthly, suppose that $\{a_1,a_2,a_3\}\subseteq U$. Without loss of generality, suppose that $a_1\in U_1$. Then for $i\in\{1, 2\}$, $v_i\sim a_2$ or $v_i\sim a_3$ to forbid a fork on \{$v_i,a_1,y,a_2,a_3$\}. Without loss of generality, let $v_1\sim a_2$. If $v_2\sim a_2$, then \{$v_4,v_3,v_2,a_1,a_2$\} induces a fork, a contradiction. Therefore, $v_2\not\sim a_2$ and $v_2\sim a_3$. By Lemma~\ref{2.4}, we have that $N_C(a_2)=\{v_1, v_n\}$ and $N_C(a_3)\in \{\{v_1, v_2\}, \{v_2, v_3\}\}$. If $N_C(a_3)=\{v_1,v_2\}$ then \{$v_4,v_3,v_2,a_1,a_3$\} induces a fork. If $N_C(a_3)=\{v_2,v_3\}$  then  \{$v_3,a_3,y,a_1,a_2$\} induces a fork. Both are contradictions. This completes the proof of Claim \ref{claw1}.
\end{proof}
\begin{claim}\label{claw3}
	For any $y\in M(C)$, $y$ is not a claw center.
\end{claim}
\begin{proof}
	Let $Y_0=U$, and let $Y_i=N_{M(C)}(Y_{i-1})\setminus \cup_{j=0}^{i-2}Y_j$ for $i=1,2,...$. Then, $M(C)=\cup_{i\geq 1}Y_i$ since $G$ is connected and $U'=\emptyset$, and for any $i\geq 1$ and $y\in Y_i$, $N(y)\subseteq Y_{i-1}\cup Y_i\cup Y_{i+1}$. To prove Claim \ref{claw3}, we will show by induction that for any $i\geq 1$ and $y\in Y_i$, $y$ is not a claw center. The statement holds when $i=1$ by Claim~\ref{claw1}. Now, let $i\geq 1$, and suppose that $y\in Y_{i+1}$ is the center of a claw on $\{a_1, a_2, a_3, y\}$.

First we consider the case that $\{a_1,a_2,a_3\}\cap Y_i\neq \emptyset$. Without loss of generality, let $a_1\in Y_i$ and $y_{i-1}\in N_{Y_{i-1}}(a_1)$.  If $i=1$, then by Lemma~\ref{lem-NMC(u)-clique}, we have that $y_{i-1}\not\sim a_2$ and $y_{i-1}\not\sim a_3$, which force an induced fork on \{$y_{i-1},a_1,y,a_2,a_3$\}, a contradiction.
If $i\geq 2$, let $a\in N_{Y_{i-1}}(y_{i-1})$, then  $y_{i-1}\not\sim a_2$ and $y_{i-1}\not\sim a_3$ since $y_{i-1}$ is not a claw center, which force an induced fork on \{$y_{i-1},a_1,y,a_2,a_3$\}, a contradiction.

Therefore, $\{a_1,a_2,a_3\}\cap Y_i=\emptyset$, and so $\{a_1,a_2,a_3\}\subseteq Y_{i+1}\cup Y_{i+2}$. Let $y_{i}\in N_{Y_{i}}(y)$ and  $y_{i-1}\in N_{Y_{i-1}}(y_{i})$. Since $y_{i-1}$ is anticomplete to $Y_{i+1}\cup Y_{i+2}$ and $y_i$ is not a claw center, we have that $|N_{\{a_1,a_2,a_3\}}(y_i)|\leq 1$. Without loss of generality, suppose that $y_i\not\sim a_2$ and $y_i\not\sim a_3$. Then, \{$y_{i-1},y_i,y_{i+1},a_2,a_3$\} induces an induced fork, a contradiction. This completes the proof of Claim \ref{claw3}.
\end{proof}

It follows from Claim~\ref{claw3} that $v$ cannot be a claw center since $v\in M(C)$, which contradicts our assumption. Therefore, $G[M(v)]$ must be perfect if $v$ is a claw center. This completes the proof of Lemma~\ref{lem-claw}.
\end{proof}

\medskip

Now, we can prove Theorem~\ref{miniclaw}.

\noindent\textbf{Proof of Theorem~\ref{miniclaw}}: Let $G$ be an MNPWD fork-free graph. Suppose that $G$ contains a claw with center $v$. Then $G[M(v)]$ is perfect by Lemma~\ref{lem-claw}, which contradicts Lemma~\ref{Mv}. This completes the proof of Theorem~\ref{miniclaw}.\qed

\section{Proof of Theorems~\ref{P_7} and \ref{P6cupK1}}\label{sec4}

In this section, we prove that (fork, $P_7$)-free graphs and (fork, $P_6\cup K_1$)-free graphs are perfectly weight divisible. All the notations $U_i, U, U', Y, Y', Z$ and $Z'$ are defined the same as in Section \ref{sec2}, which are all associated with some odd hole.

\medskip

\noindent\textbf{Proof of Theorem~\ref{P_7}}: Let $G$ be an MNPWD (fork, $P_7$)-free graph, $v\in V(G)$. By Lemma~\ref{Mv}, $G[M(v)]$ is not perfect, and thus contains an odd hole by Lemma~\ref{2.3}, say $C=v_1v_2\ldots v_nv_1$.

Let $u\in U_i$ for some $i\in\{1,2,...,n\}$. Then for any $a\in N_{M(C)}(u)$ and $a'\in N_{M(C)}(a)$, $u\sim a'$ to forbid an induced $P_7=a'auv_{i+1}v_{i+2}v_{i+3}v_{i+4}$. Therefore, 
for any $u\in U$ and any component $B$ of $G[M(C)]$, $u$ is not mixed on $V(B)$, which contradicts Lemma~\ref{lem-u-mix}. This completes the proof of Theorem~\ref{P_7}. \qed

\medskip

\noindent\textbf{Proof of Theorem \ref{P6cupK1}}: Let $G$ be an MNPWD (fork, $P_6\cup K_1$)-free graph, $v\in V(G)$. By Lemma~\ref{Mv}, $G[M(v)]$ is not perfect, and thus contains an odd hole by Lemma~\ref{2.3}, say $C=v_1v_2\ldots v_nv_1$. 

We show first that 
\begin{equation}\label{eqa-MC-Clique}
\mbox{$M(C)$ is a clique.}
\end{equation}

Let $X=Y\cup (N(Y)\cap M(C))$, and let $y_1, y_2\in X$ such that $y_1\in Y$. Without loss of generality, let $u_1\sim y_1$. If $y_1\not\sim y_2$, then $u_1\not\sim y_2$ by Lemma~\ref{lem-NMC(u)-clique}, and thus $\{y_1, u, v_2, v_3, v_4, v_5, y_2\}$ induces a $P_6\cup K_1$, a contradiction. This shows that $Y$ is a clique and is complete to $N(Y)\cap M(C)$. If $N(Y)\cap M(C)$ is not a clique, let $z_1, z_2\in N(Y)\cap M(C)$ such that $z_1\not\sim z_2$, and choose $y\in Y$ with a neighbor $u\in U$, then $y$ is the center of a claw on $\{u, y, z_1, z_2\}$, a contradiction to Theorem~\ref{miniclaw}. Therefore, $X$ is a clique.

Let $Q$ be a maximum clique of $G[M(C)]$ such that $X\subseteq Q$, and suppose, without loss of generality, that $Q\ne M(C)$. Since $G$ is connected, and since $U'=\emptyset$ by Lemma~\ref{lem-MNPD-parachute}, we may choose $z\in M(C)\setminus Q$, $z_1\in Y$ and $z_2\in Q$ such that $z\not\sim z_1$ and $z\sim z_2$. Let $u_i\in U_i$ be a neighbor of $z_1$. Then, $\{z_1, u_i, v_{i+1}, v_{i+2}, v_{i+3}, v_{i+4}, z\}$ induces a $P_6\cup K_1$, a contradiction. Therefore, $Q=M(C)$, and (\ref{eqa-MC-Clique}) holds.

By Lemma~\ref{lem-u-mix}, we may choose a vertex, say $v_0$, of $G$, an odd hole $C_0$ contained in $G[M(v_0)]$, and some $u\in U_i(C_0)$ such that $u$ is mixed on $M(C_0)$ (notice here $M(C_0)$ is a clique by (\ref{eqa-MC-Clique})). Then, $U_i(C_0)\cup M(C_0)$ is not a clique. If $G[M(U_i(C_0)\cup M(C_0))]$ is not perfect, then it has an odd hole $C'$ such that $U_i(C_0)\cup M(C_0)\subseteq M(C')$, which contradicts (\ref{eqa-MC-Clique}). Therefore, $G[M(U_i(C_0)\cup M(C_0))]$ is perfect. Since $U_i(C_0)$ is a clique by (\ref{eqa-Ui-clique}), and since $M(C_0)$ is a clique by (\ref{eqa-MC-Clique}), we have that $U_i(C_0)\cup M(C_0)$ is perfect.  Let $W=N(U_i(C_0)\cup M(C_0))$. By Lemma~\ref{lem-W}, 
$(W, V(G)\setminus W)$ is a perfect weighted division of $G$ for any positive integer function, a contradiction to the choice of $G$. This completes the proof of Theorem~\ref{P6cupK1}.\qed

\end{document}